\documentclass[12pt]{amsart}

\usepackage{amssymb,amsthm}

\newtheorem{theorem}{Theorem}[section]
\newtheorem{corollary}[theorem]{Corollary}

\newtheorem{lemma}[theorem]{Lemma}

\def\ker#1{{\rm ker} (#1)}

\begin{document}

\title[Semi-extraspecial groups]{Semi-extraspecial groups with an abelian subgroup of maximal possible order}
\author[Mark L. Lewis]{Mark L. Lewis}

\address{Department of Mathematical Sciences, Kent State University, Kent, OH 44242}
\email{lewis@math.kent.edu}


\begin{abstract}
Let $p$ be a prime.  A $p$-group $G$ is defined to be semi-extraspecial if for every maximal subgroup $N$ in $Z(G)$ the quotient $G/N$ is a an extraspecial group.  In addition, we say that $G$ is ultraspecial if $G$ is semi-extraspecial and $|G:G'| = |G'|^2$.  In this paper, we prove that every $p$-group of nilpotence class $2$ is isomorphic to a subgroup of some ultraspecial group.  Given a prime $p$ and a positive integer $n$, we provide a framework to construct of all the ultraspecial groups order $p^{3n}$ that contain an abelian subgroup of order $p^{2n}$.  In the literature, it has been proved that every ultraspecial group $G$ order $p^{3n}$ with at least two abelian subgroups of order $p^{2n}$ can be associated to a semifield.  We provide a generalization of semifield, and then we show that every semi-extraspecial group $G$ that is the product of two abelian subgroups can be associated with this generalization of semifield.
\end{abstract}

\maketitle

\section{Introduction}

Let $p$ be a prime, and let $G$ be a $p$-group.  We say that $G$ is {\it semi-extraspecial} if for every subgroup $N$ having index $p$ in $Z(G)$, then $G/N$ is an extraspecial group.  This definition seems to have originated by Beisiegel in \cite{beis}.  Also following Beisiegel, we say that a semi-extraspecial group $G$ is {\it ultraspecial} if $|G'| = |G:G'|^{1/2}$.  If $G$ is an ultraspecial group, then there is a positive integer $n$ so that $|G| = p^{3n}$.  We will say that $G$ is ultraspecial group of degree $n$.  We will often refer to semi-extraspecial groups as s.e.s. groups.

In our recent expository paper, \cite{semi}, we collect many of the known results regarding semi-extraspecial and ultraspecial groups and present them in a unified fashion.  We refer the reader to that paper for background and references for these groups.  For this paper, one key result that was presented in \cite{semi} is the following result of Verardi that any abelian subgroup of a s.e.s. group $G$ has order at most $|G:G'|^{1/2} |G'|$ (Theorem 1.8 of \cite{ver}).  Notice that when $G$ is an ultraspecial group, this bound is $|G:G'|$.  When we mention an abelian subgroup of maximal possible order, we mean an abelian subgroup of $G$ which has order $|G:G'|^{1/2} |G'|$ when $G$ is a s.e.s. group and $|G'|$ when $G$ is ultraspecial.

In the second half of \cite{semi}, we focused on the case when $G$ is an ultraspecial group with at least two abelian subgroups of order $|G:G'|$.  That paper gives references to a number of results that show when $G$ is a ultraspecial group that has at least two abelian subgroups of order $|G:G'|$ which both have exponent $p$, then $G$ can be identified with a finite semifield of order $|G'|$.  In Section \ref{semifield groups}, we will give details on the construction of semifield groups and we will detail the relationship between semifield groups and ultraspecial groups with at least two abelian subgroups of the maximal possible order.


In the main theorem of this paper, Theorem \ref{one}, we present a generalization of the construction of semifield groups.  We will see that this construction has a number of profound consequences.  One application of this construction shows that the class of ultraspecial groups is ``large.''  This is the content of the following theorem.

\begin{theorem} \label{maina}
Let $p$ be an odd prime.  If $H$ is a $p$-group of nilpotence class $2$ and exponent $p$, then there is an ultraspecial group $G$ that has a subgroup isomorphic to $H$.
\end{theorem}

We also use Theorem \ref{one} to construct ultraspecial groups $G$ having exactly one abelian subgroup of order $|G:G'|$.  We note that our construction is related to Verardi's method of constructing an ultraspecial group with one abelian subgroup from an ultraspecial $p$-group with exponent $p$ and having at least two abelian subgroups of order $|G:G'|$ that appears in Section 4 of \cite{ver}.

\begin{theorem} \label{mainb}
Let $p$ be a prime.  If $n \ge 3$ is an integer, then there exists an ultraspecial group of order $p^{3n}$ having only one abelian subgroup of order $p^{2n}$.
\end{theorem}

Also, we will show that every ultraspecial group $G$ with one abelian subgroup of order $|G:G'|$ is associated with a unique semifield up to isotopism.  (We will define isotopism in Section \ref{semifield groups}.)

We also use our construction in a different direction.  To generalize semifield groups, we change our view of semifields.  Instead of viewing a semifield as a set with two binary operations, we instead view a semifield as a nonsingular map from a direct product of additive group to itself.  Because this definition is technical, we do not present it here, but we will give a formal definition in Sections \ref{maps}.  We will then define a generalized nonsingular map where we allow the image to be a different additive group.  With this definition we are able to prove the following.

\begin{theorem} \label{mainc}
Let $V$ and $W$ be elementary abelian $p$-groups of orders $p^n$ and $p^m$ respectively.  If $\alpha : V \times V \rightarrow W$ is a generalized nonsingular map, then we can associate a unique s.e.s. group $G = G (\alpha)$ to $\alpha$ so that $|G:G'| = p^{2n}$, $|G'| = p^m$ , $G$ has abelian subgroups $A$ and $B$ of order $p^{n+m}$ so that $G = AB$, and $A$ and $B$ are elementary abelian $p$-groups.  If $p$ is odd, then $G$ has exponent $p$.
\end{theorem}

We can also show that every s.e.s. group that is generated by two elementary abelian subgroups can be obtained this way.

\begin{theorem} \label{maind}
Let $G$ be a s.e.s. group with $|G:G'| = p^{2n}$ and $|G'| = p^m$ and exponent $p$.  If $G$ contains elementary abelian subgroups $A$ and $B$ so that $G = AB$, then there exists a generalized nonsingular map $\alpha : V \times V \rightarrow W$ where $V$ and $W$ are elementary abelian $p$-groups of orders $p^n$ and $p^m$ respectively so that $G \cong G(\alpha)$.
\end{theorem}

When we modify our construction to incorporate a bilinear map $\beta : V \times V \rightarrow W$ in addition to the generalized nonsingular map $\alpha : V \times V \rightarrow W$ where $V$ and $W$ are elementary abelian $p$-groups of orders $p^n$ and $p^m$, we obtain a s.e.s. group $G = G(\alpha,\beta)$ where $|G:G'| = p^{2n}$, $|G'| = p^m$, $G$ has an elementary abelian subgroup $A$ of order $p^{n+m}$, and when $p$ is odd, $G$ has exponent $p$.   We will show that if $G$ is any s.e.s. group with $|G:G'| = p^{2n}$, $|G'| = p^m$, and exponent $p$ so that $G$ has an abelian subgroup $A$ of order $p^{n+m}$, then there exist $\alpha$ and $\beta$ as above so that $G \cong G (\alpha,\beta)$.

We will show in Section \ref{two abelian} that ideas of isotopism and anti-isotopism can be extended to these generalized nonsingular maps.  In Theorem 5.1 of \cite{hir} and Theorem 6.6 of \cite{knst1} (see also Theorem 9.1 of \cite{semi}), it is shown two semifield groups are isomorphic if and only if the semifields are isotopic or anti-isotopic.  We will generalize this result as follows.

\begin{theorem} \label{maine}
Let $V$ and $W$ be elementary abelian $p$-groups of order $p^n$ and $p^m$ respectively where $m > n/2$, and let $\alpha_1, \alpha_2 : V \times V \rightarrow W$ be generalized nonsingular maps.  Then $G (\alpha_1) \cong G (\alpha_2)$ if and only if $\alpha_1$ and $\alpha_2$ are either isotopic or anti-isotopic.
\end{theorem}

Furthermore, in Proposition 4.2 of \cite{hir}, Lemma 4.3 of \cite{knst1}, and Theorem 3.14 of \cite{ver} (see also Theorem 10.1 of \cite{semi}), it is proved that a semifield group has more than two abelian subgroups of the maximal order if and only if the semifield that is isotopic to commutative semifield.  Translating to generalized nonsingular maps, we will define symmetric maps in Section \ref{complements}, and we obtain the following result.

\begin{theorem} \label{mainf}
Let $V$ and $W$ be elementary abelian $p$-groups of order $p^n$ and $p^m$ respectively, and let $\alpha : V \times V \rightarrow W$ be a generalized nonsingular map.  Then $G (\alpha)$ has three distinct abelian subgroups $A, B, C$ so that $G = AB = AC = BC$ if and only if $\alpha$ is isotopic to a symmetric generalized nonsingular map.
\end{theorem}




\section{semifield groups}\label{semifield groups}

We say $(F,+,*)$ is a {\it pre-semifield} if $(F,+)$ is an abelian group with at least two elements whose identity is $0$ and $*$ is a binary operation so that $a*(b+c) = a*b + a*c$ and $(a+b)*c = a*c + b*c$ for all $a,b,c \in F$ and $a*b = 0$ implies that either $a = 0$ or $b = 0$.  If $F$ has an identity under $*$, then we say that $F$ is a {\it semifield}.


Given a (pre)-semifield $(F,+,*)$, we can define the {\it semifield group} $G (F,*)$ to be the group with the set $\{ (a,b,c) \mid a,b,c \in F \}$ where the multiplication is given by $(a_1,b_1,c_1) (a_2,b_2,c_2) = (a_1 + a_2,b_1 + b_2, c_1 + c_2 + a_1 * b_2)$.  When the multiplication is clear, we will drop the $*$ and just write $G (F)$ for $G (F,*)$.  Observe that the subgroups $A_1 = \{ (a,0,c) \mid a,c \in F \}$ and $A_2 = \{ (0,b,c) \mid b,c \in F \}$ will be abelian subgroups of $G(F)$.  When $|F|$ is finite, it is not difficult to show that $G(F)$ will be an ultraspecial group with at least two abelian subgroups of maximal order $|F|^2$. (See Lemma 3 of \cite{beis}.)


We say that two pre-semifields $(F_1,+_1,*_1)$ and $(F_2,+_2,*_2)$ are {\it isotopic} if there exist additive group isomorphisms $\alpha, \beta, \gamma: F_1 \rightarrow F_2$ so that $\gamma (a *_1 b) = \alpha (a) *_2 \beta (b)$ for all $a, b \in F_1$.
It has been shown that being isotopic is an equivalence relation on semifields and pre-semifields, and also, that every pre-semifield is isotopic to some semifield.  If $*$ is associative, then $F$ is a field, and when $|F|$ is finite, this semifield is obviously isomorphic to the unique field with order $|F|$.  We say that the associated semifield group $G (F)$ is the {\it Heisenberg group} of order $|F|$.  In particular, we reserve the name Heisenberg group for $G (F)$ when $F$ is a field.  We note that in some places in the literature the name Heisenberg groups is used for $G (F)$ when $F$ is any semifield.   Any semifield that is isotopic to a field will have an associative multiplication, and thus, will be isomorphic to that field.

We now suppose that $F$ is an arbitrary (pre)-semifield.  The group $G (F)$ will have more than two abelian subgroups of maximal order if and only if $F$ is isotopic to a commutative semifield this was proved by Verardi as Theorem 3.14 of \cite{ver} when $p$ is odd, and by Hiranime as Proposition 4.2 (i) of \cite{hir} and Knarr and Stroppel as Lemma 4.3 of \cite{knst1} for all primes including $p = 2$.


Let $F$ be a (pre)-semifield, then we define $*^{\rm op}$ by $a *^{\rm op} b = b * a$.  It is not difficult to see that $F^{\rm op} = (F,+,*^{\rm op})$ is a (pre)-semifield.  Obviously, if $F$ is commutative, then $F = F^{\rm op}$.  On the other hand, it is possible to have $F$ isotopic to $F^{\rm op}$ when $F$ is not isotopic to a commutative semifield.  When $F$ is isotopic to $F^{\rm op}$, we say that $F$ is {\it self-dual}.  We say that $F_1$ and $F_2$ are {\it anti-isotopic} if $F_1$ is isotopic to $F_2^{\rm op}$.

It is not difficult to show that $F$ and $F^{\rm op}$ are isotopic if and only if $F$ is anti-isotopic to itself.  In Theorem 5.1 of \cite{hir} and Theorem 6.6 of \cite{knst1}, it is proved that $G (F_1)$ and $G(F_2)$ are isomorphic if and only if $F_1$ and $F_2$ are either isotopic or anti-isotopic.  It is quite clear that $F$ and $F^{\rm op}$ are always anti-isotopic, so $G(F) \cong G(F^{\rm op})$ whether or not $F$ is self-dual.

\section{Generalized semifield groups} \label{maps}

Let $V$ be a finite additive group.  Suppose that $\alpha : V \times V \rightarrow V$ is a biadditive map so that $\alpha (v_1,v_2) = 0$ implies either $v_1 = 0$ or $v_2 = 0$.  It is not very difficult to see that defining a multiplication $*_\alpha$ by $v_1 *_\alpha v_2 = \alpha (v_1,v_2)$ will make $(V,+,*_\alpha)$ a pre-semifield.  Conversely, if $(F,+,*)$ is a (pre)-semifield, then by defining $\alpha_F (a,b) = a*b$, we see that $\alpha_F : F \times F \rightarrow F$ is a biadditive map so that $\alpha_F (a,b) = 0$ implies $a = 0$ or $b = 0$.   With this in mind, we will say that $\alpha$ is a {\it nonsingular map} if $\alpha : V \times V \rightarrow V$ is biadditive so that $\alpha (v_1,v_2) = 0$ implies either $v_1 = 0$ or $v_2 = 0$.   

Suppose that $\alpha$ is a nonsingular map.  It is not difficult to see that the property that $\alpha (v_1,v_2) = 0$ implies that $v_1 = 0$ or $v_2 = 0$ is equivalent to saying when $v_1 \ne 0$ that the kernel of the map $v_2 \mapsto \alpha (v_1,v_2)$ is $0$ and when $v_2 \ne 0$ that the kernel of the map $v_1 \mapsto \alpha (v_1,v_2)$ is $0$.  This says that when $\alpha$ is a nonsingular map, the map $v_2 \mapsto \alpha (v_1,v_2)$ is one-to-one when $v_1 \ne 0$ and the map $v_1 \mapsto \alpha (v_1,v_2)$ is one-to-one when $v_2 \ne 0$.  Since $V$ is finite, this implies that that the map $v_2 \mapsto \alpha (v_1,v_2)$ is onto when $v_1 \ne 0$ and $v_1 \mapsto \alpha (v_1,v_2)$ is onto when $v_2 \ne 0$.  In particular, if  $\alpha : V \times V \rightarrow V$ is a biadditive map, then $\alpha$ is a nonsingular map if and only if $\alpha (v,V) = V$ and $\alpha (V,v) = V$ for all $0 \ne v \in V$.  Thus, we obtain the following generalization: let $V$ and $W$ be additive groups.  We say that $\alpha : V \times V \rightarrow W$ is a {\it generalized nonsingular map} if $\alpha$ is biadditive and $\alpha (v,V) = \alpha (V,v) = W$ whenever $v \ne 0$ for $v \in V$.  Note that when we apply these definitions $V$ and $W$ will be elementary abelian $p$-groups for some prime $p$, and so, they may be viewed as vector spaces over the field of order $p$ and we may assume that $\alpha$ is a bilinear map.

We now generalize the semifield groups from Section \ref{semifield groups} in two ways.  First, we replace the semifield with a generalized nonsingular map.  Second, we add a second biadditive map $\beta$ that only involves elements of the second coordinate.  If $\beta : V \times V \rightarrow W$ is a biadditive map, then we define $\overline {\beta} (b_1,b_2) = \beta (b_1,b_2) - \beta (b_2,b_1)$ for all $b_1, b_2 \in V$. 

\begin{theorem} \label{one}
Let $p$ be a prime, let $m \le n$ be integers, and let $V$ and $W$ be elementary abelian groups of orders $p^n$ and $p^m$ respectively viewed additively.  Let $\alpha : V \times V \rightarrow W$ be a generalized nonsingular map and let $\beta : V \times V \rightarrow W$ be a biadditive map.  Consider the set $G = V \times V \times W$ and define a multiplication on $G$ by
$$
(a_1,b_1,c_1) (a_2,b_2,c_2) = (a_1+a_2,b_1+b_2, c_1+c_2+ \alpha (a_1,b_2) + \beta (b_1,b_2)).
$$
Then the following hold:
\begin{enumerate}
\item $G$ is a semi-extraspecial group with $G' = Z (G) = \{ (0,0,c) \mid c \in W \}$.
\item $[(a_1,b_1,c_1),(a_2,b_2,c_2)] = (0,0,\alpha (a_1,b_2) - \alpha (a_2,b_1) + \overline {\beta} (b_1,b_2))$.
\item If $p$ is odd, $G$ has exponent $p$.
\item $A = \{ (a,0,c) \mid a \in V, c \in W \}$ is an abelian subgroup of order $p^{n+m}$.
\item $B = \{ (0,b,c) \mid b \in V, c \in W \}$ is a subgroup of order $p^{n+m}$.
\item $G = AB$ and $A \cap B = G'$.
\item For an element $v \in V \setminus \{ 0 \}$, we have $\overline\beta (u,v) = 0$ for all $u \in V$ if and only if $(0,v,c) \in Z (B)$ for all $c \in W$.  In particular, $B \le C_G (0,v,c)$ if and only if $\overline\beta (u,v) = 0$ for all $u \in V$.
\item $\overline\beta (u,v) = 0$ for all $u,v \in V$ if and only if $B$ is abelian.
\end{enumerate}
\end{theorem}

\begin{proof}
It is a not difficult, but somewhat tedious calculation to see that the multiplication is associative.  Also, it is not difficult to see that $(0,0,0)$ is an identity and
$$(a,b,c)^{-1} = (-a,-b,-c + \alpha (a,b) + \beta (b,b)).$$
Thus, $G$ is a group.  In addition, one can easily complete the computation to see that conclusion (2) holds.  Using conclusion (2), we have that $Z(G) \ge C = \{ (0,0,c \mid c \in W \}$ and $G' \le C$.

Note that one of the many conditions equivalent to a $p$-group $G$ being semi-extraspecial is that $G$ is special and for every element $g \in G \setminus G'$ and every element $z \in G'$ there exists an element $g' \in G$ so that $[g,g'] = z$ (see Theorem 3.1 of \cite{ChMc} or Theorem 5.5 of \cite{semi}).  Thus, to show that $G$ is semi-extraspecial and that $Z(G) = G' = C$, we show for every element $g = (a,b,c) \in G \setminus C$ and every element $(0,0,z) \in C$ that there exists an element $g' = (a',b',c') \in G$ so that $[g,g'] = (0,0,c)$.  Suppose $b \ne 0$.  Since $\alpha (V,b) = W$, there exists $a' \in V$ so that $\alpha (a',b) = -z$.  Taking $b'=0$ and $c' = 0$, we obtain
$$[g,g'] = (0,0,\alpha (a,0) - \alpha (a',b) + \overline{\beta} (b,0)) = (0,0,0-(-z)+0) = (0,0,z).$$
Now, assume $b = 0$.  Since $\alpha (a,V) = W$, there exists $b' \in V$ so that $\alpha (a,b') = z$.  Taking $a' = 0$ and $c' = 0$, we see that
$$
[g,g'] = (0,0,\alpha (a,b') - \alpha (0,0) + \overline{\beta} (0,b')) = (0,0,z).
$$
This proves conclusion (1).

We prove that
$$
(a,b,c)^n = (na,nb,nc + {n \choose 2} (\alpha (a,b) + \beta (b,b)))
$$
by induction on $n$.  Observe that
$$
(a,b,c)^2 = (2a,2b,2c + \alpha (a,b) + \beta (b,b)).
$$
Since ${2 \choose 2} = 1$, the base case holds.  By induction,
$$
(a,b,c)^{n-1} = ((n-1)a,(n-1)b, (n-1)c +  {{n-1} \choose 2} (\alpha (a,b) + \beta (b,b))).
$$
We see that the first and second coordinates of $(a,b,c)^n$ will be
$$(n-1)a+a = na \ \ \ \ ~~~~{\rm and} ~~~~\ \ \ \ (n-1)b + b = nb.$$
The third coordinate of $(a,b,c)^n$ will be
$$(n-1)c + {{n-1} \choose 2} (\alpha (a,b) + \beta (b,b)) + c + \alpha((n-1)a,b) + \beta ((n-1),b).$$
Observe that
$${{n-1} \choose 2} \alpha (x,y) = \frac {(n-1)(n-2)} 2 \alpha (x,y)$$
and $\alpha ((n-1)x,y) = (n-1) \alpha (x,y)$ for all $x,y$.  Also,
$$\frac {(n-1)(n-2)}2 + \frac {(n-1)2}2  = \frac {(n-1)n}2 = {n \choose 2}.$$
It follows that
$$
{{n-1} \choose 2} \alpha (x,y) + \alpha ((n-1)x,y) = {n \choose 2} \alpha (x,y)
$$
for all $x,y$.  Using these observations, we determine that the third coordinate of $(x,y,z)^n$ is $nc + {n \choose 2} \alpha (a,b) + {n \choose 2} \beta (b,b))$.  When $p$ is odd, we know $p$ divides ${p \choose 2}$, and so conclusion (3) holds.

It is easy to see that $A$ is an abelian subgroup of $G$ and $|A| = |V||W| = p^{n+m}$ so this is conclusion (4).  Similarly, $B$ is a subgroup of $G$ and $|B| = |V||W| = p^{n+m}$ yielding conclusion (5).  Conclusion (6) is immediate.  Notice that $(0,v,c)$ will commute with every element $(0,u,c') \in B$ if and only if $\overline {\beta} (v,u) = 0$ for all $u \in V$.  This implies that $\beta (v,u) = \beta (u,v)$ for all $u \in V$ if and only if $(0,v,c) \in Z(B) \setminus Z(G)$.  This yields conclusion (7).  Observe that conclusion (8) is an immediate consequence of conclusion (7).
\end{proof}

Given a generalized nonsingular map $\alpha : V \times V \rightarrow W$ and a bilinear map $\beta : V \times V \rightarrow W$, we write $G (V,W,\alpha,\beta)$ for the group $G$ in Theorem \ref{one}.  When $V$ and $W$ are clear, we will write $G (\alpha,\beta)$ in place of $G(V,W,\alpha,\beta)$.

If $\overline\beta = 0$, then $A$ and $B$ will be two abelian subgroups of order $|V||W|$ in $G = G (\alpha,\beta)$ whose product is $G$ and whose intersection is $G'$.  Notice that if $V = W$, then $\alpha$ is a nonsingular map.  If both $V = W$ and $\overline {\beta} = 0$, then $G (\alpha,\beta)$ is the semifield group associated with the (pre)-semifield given by $\alpha$.  In all cases, we will write $G (\alpha)$ to denote the group $G (\alpha,0)$.  Also, note that this proves Theorem \ref{mainc}.  If $F$ is a pre-semifield and $\alpha_F$ is the associated nonsingular map determined by $F$, then $G (F) = G (\alpha_F)$.  Thus, we can view the groups $G(\alpha,\beta)$ as generalizations of semifield groups, and we call them {\it generalized semifield groups}.

The situation of generalized semifield groups will arise in a number of places, so we set the following hypotheses:

{\bf Hypothesis 1:}  Let $p$ be a prime.  Let $V$ and $W$ be elementary abelian $p$-groups of orders $p^n$ and $p^m$ respectively viewed additively.  Let $\alpha: V \times V \rightarrow W$ be a generalized nonsingular map. Let $\beta : V \times V \rightarrow W$ be a biadditive map.  Let $A$ and $B$ be the subgroups of $G (\alpha, \beta)$ that are defined in Theorem \ref{one}.

\section{Proof of Theorem \ref{maina}}

In this section, we will prove Theorem \ref{maina}.  We will actually prove a stronger result.  To do this we need the following well-known definition of Hall.

Recall that two groups $G$ and $H$ are {\it isoclinic} if there exist isomorphisms $a : G/Z(G) \rightarrow H/Z(H)$ and $b : G' \rightarrow H'$ so that
$$[a(g_1 Z(G)),a(g_2 Z(G))] = b ([g_1, g_2])$$
for all $g_1, g_2 \in G$.   It is not difficult to show that isoclinic determines an equivalence relation on groups.  It is easy to see that if $G$ and $H$ are isomorphic, then $G$ and $H$ are isoclinic.  On the other hand, it is well known  that if $G$ and $H$ are two extraspecial groups of the same order, then $G$ and $H$ are isoclinic.  Since $G$ and $H$ need not be isomorphic, being isoclinic is weaker than being isomorphic.


Let $p$ be an odd prime, and let $P$ be a $p$-group with a subgroup $X$ so that $X$ is central in $P$, $Z(P)$ and $P/X$ are elementary abelian, $|X| = p^m$, $|P:X| = p^n$, and $n \ge m$.   Take $V$ and $W$ to be vector spaces of dimensions $n$ and $m$ respectively over $Z_p$, the field of order $p$.  We can find linear isomorphisms $\delta : V \rightarrow P/X$ and $\tau : X \rightarrow W$.  Since $p$ is odd, we know that $2$ has a unique multiplicative inverse in $Z_p$, and we write $2^{-1}$ for this element.  We define the bilinear map $\beta_P : V \times V \rightarrow W$ by $\beta (v_1,v_2) = 2^{-1} \tau ([\delta (v_1),\delta (v_2)])$ for all $v_1, v_2 \in V$.  Let $\alpha$ be any generalized nonsingular map from $V \times V$ to $W$.  Since $m \le n$, we know that there exists such an $\alpha$.  Take $G = G(\alpha,\beta_P)$, and let $B$ be the subgroup $B$ found in the conclusion of Theorem \ref{one}.  It is not difficult to see that $P$ will be isoclinic to the subgroup $B$.

Using the Universal Coefficients Theorem (see Chapter 5 of \cite{warfield}), one can show that if $P$ and $Q$ are $p$-groups with exponent $p$, then $P$ is isomorphic to $Q$ if and only if $P$ is isoclinic to $Q$.  In particular, in the situation of the previous paragraph, if $P$ has exponent $p$, then $H$ is isomorphic to $B$ since $B$ necessarily has exponent $p$.

We are now ready to prove Theorem \ref{one}.  We have that $H$ is a $p$-group with nilpotence class two and exponent $p$.  If $|H:Z(H)| = |Z(H)|$ take $P = H$ and $X = Z(H) = Z(P)$.  If $|H:Z(H)| < |Z(H)|$, then let $A$ be an elementary abelian group of order $|Z(H)|^2/|H|$.  In this case, take $P = H \times A$ and $X = Z(H) \le Z(P)$.  When $|H:Z(H)| > |Z(H)|$, we take $A$ to be an elementary abelian group of order $|H|/|Z(H)|^2$.  We take $P = A \times H$ and $X = Z (P) = A \times Z(H)$.  Notice that in all cases, we have $|P:X| = |X|$.  Let $V$ and $W$ be elementary abelian $p$-groups of order $|P:X| = |X|$.  Finally, let $\alpha$ be a generalized nonsingular map from $V \times V$ to $W$. Using the previous two paragraphs, we see that we now have that $H$ is isomorphic to a subgroup of an ultraspecial group.  This proves Theorem \ref{maina}.

\section{Complements modulo $G'$} \label{complements}

We now work to determine which choices for $\beta$ imply that $G (\alpha,\beta)$ will be the product of two abelian subgroups.  We start at the level of complements in a vector space.  The work in this section is probably well-known, but include it here to make the argument more self-contained.

\begin{lemma} \label{two}
Let $V$ be a finite additive $p$-group for some prime $p$, let $H = V \oplus V$, and let $A = \{ (v,0) \mid v \in V \}$.  Then the following are true:
\begin{enumerate}
\item If $f$ is an additive map from $V$ to $V$, then
$$B_f = \{ (f(v),v) \mid v \in V \}$$
is a complement for $A$ in $H$.
\item If $C$ is a complement for $A$ in $H$, then $C = B_f$ for some additive map $f : V \rightarrow V$.
\item $B_f = B_g$ if and only if $f = g$, for additive maps $f,g : V \rightarrow V$.
\end{enumerate}
\end{lemma}

\begin{proof}
It is easy to see that $A \cap B_f = 0$ and $|B_f| = |B|$, so $|AB_f| = |A||B_f| = p^n p^n = p^{2n} = |H|$.  We deduce that $H = AB_f$, and so, $B_f$ is a complement for $A$ in $H$.  Let $B = \{ (0,v) \mid v \in V \}$.  If $0$ is the map sending every element of $V$ to $0$, then $B = B_0$, so $B$ is a complement for $A$ in $H$.  We know that every complement for $A$ in $H$ is a transversal for $A$ in $H$.  Let $C$ be a complement for $A$ in $H$.  Then for each element $b \in B$, we see that $C \cap A + b$ will be a single element.  Thus, we can define a function $f : V \rightarrow V$ so that $C \cap A +b = \{ (f(v),v) \}$ where $b = (0,v)$.  It follows that $C = \{ (f(v),v) \mid v \in V \}$.  Since both $(f(v_1) + f(v_2), v_1 + v_2)$ and $(f (v_1 + v_2), v_1 + v_2)$ lie in $C$ and in the same coset of $A$, we deduce that $f(v_1) + f(v_2) = f (v_1 + v_2)$, and hence, $f$ is an additive map.  We conclude that $C = B_f$.  Obviously, if $f = g$, then $B_f = B_g$.  On the other hand, if $B_f = B_g$, then $(f(v),v)$ and $(g(v),v)$ are in the coset $A + (0,v)$ and lie in the same transversal, so they must be equal.  This implies that $f(v) = g(v)$ for all $v \in V$, and so, $f = g$.
\end{proof}

We now apply Lemma \ref{two} to the quotient of $G = G (\alpha,\beta)$ by its center.  If we assume Hypothesis 1, then $G/G' = A/G' \oplus B/G'$, so we can apply the notation of Lemma \ref{two} to $G/G'$.  Observe that $A/G' = \{ (v,0,0)G' \mid v \in V \}$ and $B/G' = \{ (0,v,0)G' \mid v \in V \}$.  If $f$ is an additive map from $V$ to $V$, we can then define $B_f = \{ (f(v),v,0) \mid v \in V \}$.  It is not difficult to see that $B_f$ will be a subgroup of $G$ of order $p^{n+m}$.  Also, it is not difficult to see that $B_f/G'$ will correspond to the subgroup labeled as $B_f$ when Lemma \ref{two} is applied to $G/G'$, and we will use the same notation to denote both the subgroup of $G$ and its quotient in $G/G'$.  We believe the meaning of the notation will always be clear from context.  In this next lemma, we gather some facts about $B_f$ as a subgroup of $G$.

\begin{lemma} \label{three}
Assume Hypothesis 1 with $G = G(\alpha,\beta)$.  Then the following are true:
\begin{enumerate}
\item If $f : V \rightarrow V$ is an additive map, then $G = A B_f$ and $A \cap B_f = G'$.
\item If $D \le G$ satisfies $G = AD$ and $A \cap D = G'$, then $D = B_f$ for some unique additive map $f :V \rightarrow V$.
\item Finally,
$$[(f(v_1),v_1,w_1),(f(v_2),v_2,w_2)] = $$
$$ = (0,0,\alpha (f(v_1),v_2) - \alpha (f(v_2),v_1) + \overline \beta (v_1,v_2) )$$
for all elements $v_1,v_2 \in V, w_1, w_2 \in W$ determines the commutation for elements in $B_f$ where $f:V \rightarrow V$ is an additive map.
\end{enumerate}
\end{lemma}

\begin{proof}
Observe that the first two conclusions follow from Lemma \ref{two} applied in $G/G'$.  The third conclusion arises from Theorem \ref{one} (2).
\end{proof}

Notice that Lemma \ref{three} (2) implies that if $G = G (\alpha,\beta)$, then there is a bijection between the set of additive maps from $V$ to $V$ and
$${\rm comp} (G,A) = \{ D \le G \mid G = DA, D \cap A = G' \}$$
defined by $f \mapsto B_f$, the subgroup of $G$.  The map $0_V : V \rightarrow V$ defined by $0_V (v) = 0$ is an additive map from $V$ to $V$, and $B_{0_V} = B$ from above.  Thus, when $\overline{\beta} = 0$, then $B_{0_V}$ is abelian.

When $\alpha : V \times V \rightarrow W$ is a biadditive map, we say that $\alpha$ is {\it symmetric} if $\alpha (v_1,v_2) = \alpha (v_2,v_1)$ for all $v_1, v_2 \in V$.  Observe that if $F$ is a semifield, then $F$ is commutative if and only if $\alpha_F$ is symmetric.  Recall that a semifield group $G (F)$ has more than two abelian subgroups if and only if $F$ is isotopic to a commutative semifield.  Notice that if we consider $G (F) = G(\alpha_F)$ in view of Theorem \ref{one} and fix the subgroup $A$ as in that theorem, this would say that ${\rm comp} (G(F),A)$ contains at least two abelian subgroups if and only if $F$ is isotopic to a commutative semifield.  We know that distinct abelian subgroups of maximal order in a semifield group must have a product that is the whole group.  (See Theorem 1.9 of \cite{ver}.) Hence, if $B$ and $C$ are distinct abelian subgroups that lie in ${\rm comp} (A)$, then $G = BC$ and $B \cap C = G'$.

When we consider $G = G (\alpha)$ where $\alpha$ is a generalized nonsingular map and $A$ is the subgroup found in Theorem \ref{one}, we know that ${\rm comp} (G,A)$ contains at least one abelian subgroup.  In this next corollary, we show if $\alpha$ is symmetric, then there will be more than one and their product will be $G$.  Note that if $|W|^2 \le |V|$, then we do not know that the product of abelian subgroups of order $|V||W|$ is necessarily $G$, so this result obtains a stronger conclusion than just that ${\rm comp} (G,A)$ has at least two abelian members.  We will obtain a converse later.

\begin{corollary} \label{threea}
Assume Hypothesis 1 with $\beta = 0$.  If $\alpha$ is symmetric and $G = G(\alpha)$, then ${\rm comp} (G,A)$ contains at least two abelian subgroups $B$ and $C$ that satisfy $G = BC$ and $B \cap C = G'$.
\end{corollary}

\begin{proof}
Define $1_V : V \rightarrow V$ by $1_V (v) = v$ for all $v \in V$.  It is easy to see that $1_V$ is an additive map.  We see that
$$
\alpha (1_V(v_1),v_2) = \alpha (v_1,v_2) = \alpha (v_2,v_1) = \alpha(1_V (v_2),v_1)
$$
for all $v_1, v_2 \in V$.  Since $\beta = 0$, we have $\overline {\beta} = 0$, and we can use Lemma \ref{three} (3) to see that $C = B_{1_V}$ will be abelian.  We previously noted that $B = B_{0_V}$ is abelian.  It is easy to see that $B \cap C = G'$ so $G = BC$.  This gives the two subgroups in ${\rm comp} (G,A)$ that have the properties stated.
\end{proof}

We now can identify all of the choices for $\beta$ that imply that $G(\alpha,\beta)$ has an abelian subgroup $B \in {\rm comp} (G,A)$.

\begin{lemma} \label{four}
Assume Hypothesis 1 with $G = G(\alpha,\beta)$.  Then there exists an abelian subgroup $C \in {\rm comp} (G,A)$ if and only if there exists an additive map $f : V \rightarrow V$ so that $\overline {\beta} (v_1,v_2) = \alpha (f(v_2),v_1) - \alpha (f (v_1),v_2)$ for all $v_1, v_2 \in V$.  If this occurs, then $C = B_f$.
\end{lemma}

\begin{proof}
First, suppose there exists an additive map $f : V \rightarrow V$ so that $\overline {\beta} (v_1,v_2) = \alpha (f(v_2),v_1) - \alpha (f (v_1),v_2)$ for all $v_1, v_2 \in V$.  This implies that $\alpha (f(v_1),v_2) - \alpha (f(v_2),v_1) + \overline \beta (v_1,v_2) = 0$ for all $v_1,v_2 \in V$.  Notice we may apply Lemma \ref{three} (1) and (3) to see that this implies that $B_f$ is abelian and $G = A B_f$, so we have $C = B_f$.

Conversely, suppose that $C$ is abelian so that $G = AC$ and $A \cap C = G'$.  By Lemma \ref{three} (2), it follows that $C = B_f$ for some additive map $f : V \rightarrow V$.  Since $C$ is abelian, we use Lemma \ref{three} (3) to see that $\alpha (f(v_1),v_2) - \alpha (f(v_2),v_1) + \overline \beta (v_1,v_2) = 0$ for all $v_1,v_2 \in V$, and thus, $\overline {\beta} (v_1,v_2) = \alpha (f(v_2),v_1) - \alpha (f (v_1),v_2)$ for all $v_1, v_2 \in V$.
\end{proof}

\section{Cosets in the group of alternating maps}

Let $V$ be an additive group, and let ${\rm add} (V)$ be the set of all additive maps from $V$ to $V$.  Note that using pointwise addition, we can make ${\rm add} (V)$ a group.  When $V$ is an elementary abelian $p$-group of order $p^n$, it is not difficult to see that $|{\rm add} (V)| = (p^n)^n = p^{(n^2)}$.

Let $V$ and $W$ be additive groups.  Recall that a biadditive map $\gamma:V \times V \rightarrow W$ is {\it alternating} if $\gamma (v,v) = 0$ for all $v \in V$.  When $|W|$ is odd, it is not difficult to see that $\gamma$ being alternating is equivalent to $\gamma (v_1,v_2) = - \gamma (v_2,v_1)$ for all $v_1, v_2 \in V$.  We let ${\rm alt} (V,W)$ be the set of all alternating biadditive maps $\gamma: V \times V \rightarrow W$.  Recall that if $\beta$ is any biadditive map, then $\overline {\beta} \in {\rm alt} (V,W)$.  Also, it is not difficult to see that if $\beta \in {\rm alt} (V,W)$, then $\overline {\beta} = 2 \beta \in {\rm alt} (V,W)$.  When $|W|$ is odd and $\beta \in {\rm alt} (V,W)$, we see that $\overline{\beta} = 0$ if and only if $\beta = 0$.

Using pointwise addition, we see that ${\rm alt} (V,W)$ is a group.
%
When $V$ is elementary abelian of order $p^n$ and $W$ is elementary abelian of order $p^m$, we deduce that $|{\rm alt} (V,W)| = p^{m n (n-1)/2}$.

We continue to let $V$ and $W$ be additive groups.  Suppose that $\alpha$ is a generalized nonsingular map from $V \times V$ to $W$.  For each $f \in {\rm add}  (V)$, we define $\phi_{\alpha} (f) : V \times V \rightarrow W$ by $\phi_\alpha (f) (v_1,v_2) = \alpha (f (v_1), v_2) - \alpha (f (v_2),v_1)$.  Observe that $\phi_{\alpha} (f) \in {\rm alt} (V,W)$, so $\phi_{\alpha} : {\rm add} (V) \rightarrow {\rm alt} (V,W)$.  Also, note that $\phi_{\alpha} (f+g) = \phi_{\alpha} (f) + \phi_{\alpha} (g)$, so $\phi_{\alpha}$ is a group homomorphism.  In particular, $\phi_\alpha ({\rm add} (V))$ is a subgroup of ${\rm alt} (V,W)$.

Recall that Lemma \ref{three} (2) implies that if $G = G (\alpha,\beta)$, then there is a bijection between ${\rm alt} (V)$ and ${\rm comp} (G,A)$.  Furthermore, using Lemma \ref{three} (3), commutation in $B_f$ is given by $\phi_\alpha (f) + \overline\beta$.

\begin{lemma} \label{five}
Assume Hypothesis 1.  Then the following are true:
\begin{enumerate}
\item If $\beta_1$ is a biadditive map from $V \times V$ to $W$ that satisfies $\overline {\beta_1} \in \phi_\alpha ({\rm add} (V)) + \overline {\beta}$, then $G(\alpha,\beta) \cong G(\alpha,\beta_1)$.
\item Let $G = G(\alpha,\beta)$.  Then ${\rm comp} (G,A)$ contains abelian subgroups if and only if
$$\overline {\beta} \in \phi_\alpha ({\rm add} (V)).$$
\item If $\overline {\beta} \in \phi_\alpha ({\rm add} (V))$, then $G(\alpha,\beta) \cong G(\alpha)$ and the number of abelian subgroups in ${\rm comp} (G,A)$ equals $|\ker {\phi_\alpha}|$ where $G = G(\alpha)$.
\end{enumerate}
\end{lemma}

\begin{proof}
Suppose $\overline {\beta_1} \in \phi_\alpha ({\rm add} (V)) + \overline {\beta}$.  Then there is a map $f \in {\rm add} (V)$ so that $ \overline {\beta_1} = \phi_\alpha (f) + \overline {\beta}$.  We take $A$ and $B$ as above in $G(\alpha,\beta) = G$, and the corresponding subgroups in $G(\alpha,\beta_1)$ are $A$ and $B_f$ by Lemma \ref{three}. Since $G = AB = AB_f \cong G(\alpha,\beta_1)$, we have the desired isomorphism for conclusion (1).

Conclusion (2) follows immediately from Lemma \ref{four}.  For conclusion (3), we know that if $f \in {\rm add} (V)$, then $B_f \in {\rm comp} (G,A)$ is abelian if and only if $\phi_\alpha (f) = 0$ by Lemmas \ref{three} and \ref{four}.  This gives a bijection between abelian elements of ${\rm comp} (G,A)$ and elements of $\ker {\phi_\alpha}$.
\end{proof}

Notice that Lemma \ref{five} (1) shows that if $\overline{\beta}$ and $\overline{\beta_1}$ lie in the same coset of $\ker {\phi_\alpha}$ in ${\rm alt} (V,W)$, then $G(\alpha,\beta)$ and $G(\alpha,\beta_1)$ are isomorphic.  If $\overline{\beta} \in \phi_\alpha ({\rm add} (V))$, then this implies that $G(\alpha,\beta)$ is isomorphic to $G(\alpha,0)$.  In particular, this generalizes Lemma \ref{four} and we see that the number of abelian subgroups of $G(\alpha,0)$ whose product with $A$ is $G$ equals $|\ker{\phi_\alpha}|$ by Lemma \ref{five} (3).

If $W = V$ and $\alpha$ is commutative, Verardi showed in Corollary 5.9 of \cite{ver} that the number of abelian complements of $A$ equals the size of the middle nucleus of the semifield $(V,\alpha)$.  The middle nucleus of $(V,\alpha)$ is the set
$$\{ v \in V \mid \alpha (\alpha (u,v),w) = \alpha (u,\alpha(v,w)), ~ \forall u,w \in V \}.$$
Thus, it seems likely that there is a connection between $\ker {\phi_\alpha}$ and this middle nucleus of $(V,\alpha)$, but at this time, we have not determined this connection.

When $W$ has order $p$, we know that $G = G(\alpha,\beta)$ will be an extra-special group of order $p^{2n+1}$ where $|V| = p^n$.  This implies in this case that $\phi_\alpha ({\rm add} (V)) = {\rm alt} (V,W)$.  Notice that $|{\rm add} (V)| = p^{n^2}$ and $|{\rm alt} (V,W)| = p^{n(n-1)/2}$ in this case.  Since $n^2 - n(n-1)/2 = n(n+1)/2$, we have $|\ker{\phi_\alpha}| = p^{n(n+1)/2}$.  This shows that in $G$ there are $p^{n(n+1)/2}$ abelian subgroups whose product with $A$ gives $G$.  This implies that $G$ has at least $1 + p^{n(n+1)/2}$ abelian subgroups of order $p^{n+1}$ in $G$.  However, it is not difficult to see when $n \ge 2$ that there exist abelian subgroups of order $p^{n+1}$ in $G$ whose product with $A$ is not all of $G$, so this does not give a complete count of all of the abelian subgroups of $G$ of order $p^{n+1}$.

Let $V$ and $W$ be elementary abelian $p$-groups of order $p^n$ and $p^m$ respectively where $p$ is a prime and $n \ge m$ are positive integers.  Let $\alpha: V \times V \rightarrow W$ be a generalized nonsingular map.  If $\phi_\alpha ({\rm add} (V)) < {\rm alt} (V,W)$, then there will exist a biadditive map $\beta \in {\rm alt} (V,W) \setminus \phi_\alpha ({\rm add} (V))$.  Applying Lemma \ref{five} (2), we see that ${\rm comp} (G (\alpha,\beta),A)$ contains no abelian subgroups, and so there exist no abelian subgroups of $G$ whose product with $A$ is $G(\alpha,\beta)$.

We now show when $m \ge 3$ that there exist s.e.s. groups $G$ where $|G:G'| = p^{2n}$, $|G'| = p^m$, and $G$ has an abelian subgroup $A$ of order $p^{m+n}$ whose product with any other abelian subgroup of $G$ is proper in $G$. When $m > n/2$, we can use a result of Verardi to see that $A$ will be the only abelian subgroup of order $p^{m+n}$ in $G$.  Notice that the following result gives Theorem \ref{mainb} when $n = m$.

\begin{corollary} \label{fivea}
For every prime $p$ and for all integers $n \ge m \ge 3$, there exists an s.e.s. groups $G$ where $|G:G'| = p^{2n}$, $|G'| = p^m$, and $G$ has an abelian subgroup $A$ of order $p^{m+n}$ whose product with any other abelian subgroup of $G$ is proper in $G$.  When $m > n/2$, we have that $A$ will be the only abelian subgroup of order $p^{m+n}$ in $G$.
\end{corollary}

\begin{proof}
For every prime $p$ and integers $n \ge m \ge 3$, we need to find a generalized nonsingular map $\alpha : V \times V \rightarrow W$ so that $\phi_\alpha ({\rm add} (V)) < {\rm alt} (V,W)$.  When either $m \ge 4$ or $n \ge 4$ and $m = 3$, we have that $n (n-1)m /2 > n^2$, so $|{\rm add} (V)| < |{\rm alt} (V,W)|$ and hence, any $\alpha$ will work.  When $m = 3 = n$, we see that $|{\rm add} (V)| = p^9 =  |{\rm alt} (V,W)|$.  Hence, it suffices to find $\alpha$ so that $|\ker{\phi_\alpha}| > 1$.  To see that there exists an $\alpha$ in all cases, let $F$ be the field of order $p^n$.  Take $V$ to be the additive group of $F$, write $U$ for a subgroup of $V$ of order $p^{n-m}$, and set $W = V/U$.  We define $\alpha$ by $\alpha (v_1,v_2) = v_1 v_2 + U$ where the multiplication is the multiplication from $F$.  It is easy to see that $\alpha$ will be a nonsingular map.  If $n = m = 3$, then $U = 0$ and $W = V$.  We saw earlier that $|\ker {\phi_\alpha}|$ equals the size of the middle nucleus of $F$, but since $F$ is a field, the middle nucleus is $F$, so $|\ker {\phi_\alpha}| = |V| = p^3 > 1$.  This proves the corollary.
\end{proof}

As can be seen by the proof of Corollary \ref{fivea}, when either $m \ge 4$ or $n > m = 3$, we see that for every generalized nonsingular map $\alpha$, that $\phi_\alpha ({\rm add} (V)) < {\rm alt} (V,W)$.  When $n = m = 3$, we can find for $p = 5$, $p=7$, and probably larger primes $\alpha$'s that have $|\phi_\alpha ({\rm add})| = |{\rm alt} (V,W)|$, so not all $\alpha$'s will work.  Interestingly, we can find when $m = 2$ and $n = 4$ an $\alpha$ so that $|\phi_\alpha ({\rm add})| < |{\rm alt} (V,W)|$.

Given a generalized nonsingular map $\alpha : V \times V \rightarrow W$, we have shown that if $\hat\beta_1$ and $\hat\beta_2$ are in the same coset $\phi_\alpha ({\rm add} (V))$ in ${\rm alt} (V,W)$, then $G(\alpha,\beta_1)$ and $G (\alpha,\beta_2)$ are isomorphic, and if $\hat\beta_1$ does not lie in $\phi_\alpha ({\rm add} (V))$, then $G(\alpha,\beta_1)$ is not isomorphic to $G(\alpha,0)$.  However, in general if $\hat\beta_1$ and $\hat\beta_2$ lies in different cosets of $\phi_\alpha ({\rm add} (V))$ it need not be the case that $G(\alpha,\beta_1)$ and $G(\alpha,\beta_2)$ are not isomorphic.

In particular, when $p = 3$ and $\alpha$ is the nonsingular map coming from the field of order $3^3$, so $|V| = 3^3$.  In this case, it is not difficult to see that $|{\rm alt} (V,V)| = 3^9$ and $|\phi_\alpha ({\rm add} (V))| = 3^6$.  Hence, there are $27$ cosets in this case.  Josh Maglione has written a program in the computer algebra system Magma \cite{magma} that computes the $27$ cosets and the associated generalized semifield groups using the package eMagma \cite{package}.  Also, using Magma and eMagma, Josh was able to show that the $26$ generalized semifield groups that are not the semifield group are all isomorphic.  On the other hand, when $|V| = 3^4$, we can find semifields where the associated generalized semifield groups are not all isomorphic.

\section{Obtaining generalized nonsingular maps from groups}

We have seen how to construct s.e.s. groups with an abelian subgroup of maximal possible order.  We now show that every s.e.s. group with an abelian subgroup of maximal possible order can be be obtained this way.  We begin by determining the generalized nonsingular map for such a group.

\begin{lemma} \label{six}
Let $G$ be a semi-extraspecial group where $|Z(G)| = p^m$ and $|G:Z(G)| = p^{2n}$.  Suppose $A$ is an abelian subgroup of order $p^{n+m}$.  Then the following are true:
\begin{enumerate}
\item The map $[,] : A/Z(G) \times G/A \rightarrow Z(G)$ defined by $[aZ(G),gA] = [a,g]$ is well-defined and bilinear.
\item For every element $a \in A \setminus G'$ and every element $z \in Z(G)$, there exists an element $g \in G$ such that $[a,g] = [aZ(G),gA] =  z$.
\item For every element $g \in G \setminus A$, we have $G = A C_G (g)$.  Furthermore, for every element $z \in Z(G)$, there exists an element $a \in A$ so that $[a,g] = [aZ(G),gA] = z$.
\end{enumerate}
\end{lemma}

\begin{proof}
Fix elements $a \in A$ and $g \in G$.  Suppose $z \in Z(G)$ and $b \in A$, then $[az,gb] = [a,gb]^z [z,gb] = [a,gb] = [a,b][a,g]^b = [a,g]$ where since $A$ is abelian we have $[a,b] = 1$ and $[a,g] \in G' = Z(G)$ implies $[a,g]^b = [a,g]$.  This shows that the map is well-defined.  Suppose $a_1,a_2 \in A$, then $[a_1a_2,g] = [a_1,g]^{a_2} [a_2,g] = [a_1,g][a_2,g]$ and if $g_1,g_2 \in G$, then $[a,g_1g_2] = [a,g_2][a,g_1]^{g_2} = [a,g_2][a,g_1] = [a,g_1][a,g_2]$.  Thus, the map is bilinear.  Since $G$ is semi-extraspecial, for every element $a \in A \setminus G'$ and $z \in Z(G)$, there exist an element $g \in G$ so that $[a,g] = z$.

Suppose for $g \in G \setminus A$.  We claim that $G = A C_G (g)$.  If $m = n$, then $A = C_G (a)$ for every $a \in A \setminus Z(G)$.  It follows that $C_G (g) \cap A = G'$, and since $p^n = |G:C_G (g)| = |A:G'|$, we conclude that $G = A C_G (g)$.

We may assume that $m < n$.  Observe that $[A,g] = [A,\langle g, Z(G) \rangle]$ is a subgroup of $Z(G) = G'$, and so, it is a normal subgroup of $G$.  We claim that $Z(G) = [A,g]$.  If $[A,g] < Z(G)$, then $\langle A, g \rangle/[A,g]$ will be an abelian subgroup of $G/[A,g]$.  Notice that $G/[A,g]$ is a semi-extraspecial group and $|\langle A, g \rangle/[A,g]| = p |A:[A,g]| = p p^n |G':[A,g]|$.  We saw in the Introduction that the maximal size of abelian subgroup of $G/[A,g]$ is $|G:G'|^{1/2} |G':[A,g]| = p^n |G':[A,g]|$, so this is a contradiction.  It follows that $Z(G) = [A,g]$.  Notice that the map $a \mapsto [a,g]$ is a surjective homomorphism from $A$ to $Z(G)$ whose kernel is $C_A (g)$.  This implies that $|A:C_A (g)| = |Z(G)| = p^m$.  Since $|G:C_G (g)| = p^m$ and $C_A (g) = C_G (g) \cap A$, we conclude that $G = A C_G (g)$.  This proves the claim.

Fix the element $z \in Z (G)$.  Since $G$ is semi-extraspecial, there exists an element $x \in G$ so that $[x,g] = z$.  We can write $x = ab$ for some $a \in A$ and $b \in C_G (g)$.  We have $[x,g] = [ab,g] = [a,g]^b [b,g] = [a,g]$ where the last equality holds since $b \in C_G (g)$.  We conclude that $[a,g] = [x,g] = z$ as desired.
\end{proof}

Note that if $G$ is the group $G (\alpha,\beta)$ considered in Theorem \ref{one}, then $[(a,0,0),(a',b,0)] = (0,0,\alpha (a,b))$ for all $a, a',b \in V$.

\begin{corollary}\label{seven}
Let $G$ be a s.e.s. group where $|Z(G)| = p^m$ and $|G:Z(G)| = p^{2n}$.  Suppose $A$ is an abelian subgroup of order $p^{n+m}$.  Let $V$ be an elementary abelian $p$-group of order $p^n$ and $W$ an elementary abelian $p$-group of order $p^m$.  Let $\delta, \sigma, \tau$ be isomorphisms so that $\delta : V \rightarrow A/Z(G)$, $\sigma : V \rightarrow G/A$, and $\tau : Z(G) \rightarrow W$.  If $\alpha_{G,A} : V \times V \rightarrow W$ is defined by $\alpha_{G,A} (v,w) = \tau ([\delta (v),\sigma (w)])$, then $\alpha_{G,A}$ is a generalized nonsingular map from $V$ to $W$.

Let $B$ be a subgroup of $G$ so that $A \cap B = Z(G)$ and $G = AB$, and let $\eta : V \rightarrow B/Z(G)$ be the map obtained by composing $\sigma$ with the natural map from $G/A$ to $B/Z (G)$.  Assume $p$ is odd.  Define $\beta_B : V \times V \rightarrow W$ by $\beta_B (v_1,v_2) = 1/2 (\tau ([\eta (v_1),\eta (v_2)]))$.  Then $G$ is isoclinic to $G (\alpha_{G,A},\beta_B)$.  If $A$ and $B$ have exponent $p$, then $G$ is isomorphic to $G(\alpha_{G,A},\beta_B)$.
\end{corollary}

\begin{proof}
Fix $w \in W$ and fix $v \in V$.  Let $z = \tau^{-1} (w)$.  First, we have $\delta (v) = a Z(G)$ for some $a \in A$.  By Lemma \ref{six} (2), there exists $g \in G$ so that $[aZ(G), gA] = z$.  Let $u = \sigma^{-1} (gA)$.  Then $$\alpha_{G,A} (v,u) = \tau ([\delta (v),\sigma(u)]) = \tau ([aZ(G),gA]) = \tau (z) = w.$$
This shows that $\alpha_{G,A} (v,V) = W$.  Next, we have $\sigma (v) = hA$ for some $h \in G$.  By Lemma \ref{six} (3), there exists $b \in A$ so that $[bZ(G),hA] = z$.  Let $u' = \delta^{-1} (bZ(G))$.  Then, $$\alpha_{G,A} (u',v) = \tau ([\delta (u'),\sigma(v)]) = \tau ([bZ(G),hA]) = \tau (z) = w.$$
This shows that $\alpha_{G,A} (V,v) = W$.  This proves that $\alpha_{G,A}$ is a generalized nonsingular map.

Since $A \cap B = Z(G)$, we see that $G/Z (G) = A/Z(G) \times B/Z(G)$.  Thus, the map $\gamma:(v,w,z) \mapsto (\delta (v), \eta (w)$ is an isomorphism from $G(\alpha_{G:A},\beta_B)/Z (G(\alpha_{G:A},\beta_B)) \rightarrow A/Z(G) \times B/Z(G)$.  We claim that $(\gamma,\tau^{-1})$ is an isoclinism from $G(\alpha_{G:A},\beta_B)$ to $G$.  We have
$$
[\gamma(v_1,w_1,z_1),\gamma(v_2,w_2,z_2)] = [(\delta(v_1),\eta(w_1)),(\delta(v_2),\eta(w_2))] =
$$
$$
[\delta (v_1),\eta (w_2)] + [\eta(w_1),\delta(v_2)] + [\eta(w_1),\eta(w_2)] =
$$
$$
\tau^{-1} (\alpha_{G,A} (v_1,w_2)) - \tau^{-1} (\alpha_{G,A} (v_2,w_1)) + [\eta(w_1),\eta(w_2)].
$$
On the other hand,
$$[(v_1,\!w_1,\!z_1\!),(v_2,\!w_2,\!z_2\!)]\! = \!(0,0,\alpha_{G,A} (v_1,w_2)\! -\! \alpha_{G,A}(v_2,w_1)\! + \! \overline{\beta_B(w_1,w_2)}).$$
Notice that
$$\overline{\beta_B (w_1,w_2)}\! = \!\beta_B (w_1,w_2)\! - \!\beta_B (w_2,w_1)\! =\! \frac 12 [\eta(w_1),\eta(w_2)] \!-\! \frac 12 [\eta(w_2),\eta (w_1)].$$
Since $[\eta(w_2),\eta(w_1)] = - [\eta (w_1), \eta(w_2)]$, we get the required equality for the isoclinism.
\end{proof}

Notice that if $B$ is abelian, then $\overline{\beta_B} = 0$ and so $G(\alpha_A,\beta_B) \cong G(\alpha_A)$.  Hence, Theorem \ref{maind} may be viewed as a corollary of Corollary \ref{seven}.

\section{Two abelian subgroups whose product is $G$} \label{two abelian}

Notice that $\alpha_{G,A}$ in Corollary \ref{seven} depends on the choice of $\delta$, $\sigma$, and $\tau$.  With this in mind, we make the following definition. Let $\alpha_1$ and $\alpha_2$ be generalized nonsingular maps from $V \times V$ to $W$.  We say that $\alpha_1$ and $\alpha_2$ are {\rm isotopic} if there exist isomorphisms $a,b : V \rightarrow V$ and $c: W \rightarrow W$ so that $\alpha_2 (a(v_1),b(v_2)) = c (\alpha_1 (v_1,v_2))$ for all $v_1, v_2 \in V$.  Note that when $\alpha_1$ and $\alpha_2$ are nonsingular maps and we translate to the associated semifields this is the normal definition of isotopism of semifields.  It is not difficult to see that this definition of isotopism will yield an equivalence relation on generalized nonsingular maps.  In light of Corollary \ref{seven}, $\alpha_{G,A}$ is uniquely defined up to isotopism.  We now show that isotopic generalized nonsingular maps yields isomorphisms of generalized semifield groups.

\begin{lemma}\label{eight}
Let $V$ and $W$ be elementary abelian $p$-groups of order $p^n$ and $p^m$ respectively where $p$ is a prime and $m \le n$ are positive integers.  Let $\alpha_1$ and $\alpha_2$ be generalized nonsingular maps from $V \times V$ to $W$, and let $(a,b,c)$ be an isotopism from $\alpha_1$ to $\alpha_2$.  If $\beta_1$ and $\beta_2$ are biadditive maps from $V \times V$ to $W$ that satisfy $\beta_2 (b(v_1),b(v_2)) = c(\beta_1 (v_1,v_2))$ for all $v_1,v_2 \in V$, then the map $\gamma : G(\alpha_1,\beta_1) \rightarrow G (\alpha_2,\beta_2)$ defined by $\gamma (u,v,w) = (a(u),b(v),c(w))$ is an isomorphism of groups.
\end{lemma}

\begin{proof}
It is easy to see that $\gamma$ will be a bijection.  Consider $g_1 = (u_1,v_1,w_1), g_2 = (u_2,v_2,w_2) \in G(\alpha_1,\beta_1)$.  Observe that $\gamma (g_1 g_2)$ equals
$$
(a(u_1+u_2), b(v_1+v_2), c(w_1 + w_2 + \alpha_1 (u_1,v_2) + \beta_1 (v_1,v_2))).
$$
On the other hand, $\gamma(g_1)\gamma(g_2)$ equals
$$
(a(\!u_1\!)+a(\!u_2\!), b(\!v_1\!)+b(\!v_2\!), c(\!w_1\!)+c(\!w_2\!) \!+ \!\alpha_2 (\!a(\!u_1\!),b(\!v_2))\! + \! \beta_2 (\!b(\!v_1,b(\!v_2\!)))).
$$
The equality of these two equations follows since $a$ and $b$ are isomorphisms and in the third coordinate, we use the fact $c$ is an isomorphism along with the fact that $\alpha_1$ and $\alpha_2$ are isotopic and the equation relating $\beta_1$ and $\beta_2$.
\end{proof}

We next see that isoclinisms preserve the generalized nonsingular map coming from an abelian subgroup of maximal possible order.

\begin{lemma}\label{nine}
Let $G_1$ and $G_2$ be s.e.s. groups with $|G_i:G_i'| = p^{2n}$ and $|G_i'| = p^m$ where $p$ is a prime and $m \le n$ are positive integers.  Suppose $A_1$ is an abelian subgroup of $G_1$ of order $p^{n+m}$.  Suppose $(\sigma,\delta)$ is an isoclinism from $G_1$ to $G_2$.  If $A_2/G_2' = \sigma (A_1/G_1')$, then $A_2$ is an abelian subgroup of $G_2$ of order $p^{n+m}$.  Furthermore, $\alpha_{G_1,A_1}$ and $\alpha_{G_2,A_2}$ are isotopic.
\end{lemma}

\begin{proof}
Since $(\sigma,\delta)$ is the isoclinism from $G_1$ to $G_2$, we know that $\sigma: G_1/Z(G_1) \rightarrow G_2/Z(G_2)$ and $\delta: G_1' \rightarrow G_2'$ so that $[\sigma(\overline g), \sigma(\overline h)] = \delta ([g,h])$ for all $g, h \in G_1$.  Notice that if $a, b \in A$, then $[\sigma(\overline a), \sigma (\overline b)] = \delta ([a,b]) = \delta (1) = 1$.  It follows when $A_2/G_2' = \sigma (A_1/G_1')$ that $A_2$ must be an abelian subgroup of $G_2$.

Let $V$ and $W$ be elementary abelian groups of orders $p^n$ and $p^m$ respectively.  Fix isomorphisms $\eta_1:V \rightarrow A_1/Z (G_1)$, $\zeta_1: V \rightarrow G_1/A_1$, and $\kappa_1 : Z(G_1) \rightarrow W$.  Define $\eta_2:V \rightarrow A_2/Z(G_2)$ by $\eta_2 (v) = \sigma (\eta_1 (v))$ for all $v \in V$.  Since $\sigma$ is a homomorphism and maps $A_1$ to $A_2$, it follows that $\sigma$ maps cosets of $A_1$ to cosets of $A_2$.  Thus, we can define $\zeta_2:V \rightarrow G_2/A_2$ by $\zeta_2 (v) = \sigma (\zeta_1 (v))$.  Also, we define $\kappa_2 : Z(G_2) \rightarrow W$ by $\kappa_2 (z_2) = \kappa_1 (\delta^{-1} (z_2))$.  We see that $$\alpha_{G_2,A_2} (v_1,v_2) = \kappa_2 ([\eta_2(v_1),\zeta_2 (v_2)]) = \kappa_1 (\delta^{-1} ([\sigma (\eta_1 (v_1)),\sigma (\zeta_1 (v_2))]) = $$
$$= \kappa_1 (\delta^{-1} (\delta ([\eta_1 (v_1),\zeta_(v_2)]))) = \kappa_1 ([\eta_1 (v_1),\eta_2 (v_2)]) = \alpha_{G_1,A_1} (v_1,v_2).$$
We can now conclude that $\alpha_{G_1,A_1}$ and $\alpha_{G_2,A_2}$ are isotopic.
\end{proof}

We saw in Corollary \ref{threea} that if $\alpha$ is a commutative generalized nonsingular map, then $G(\alpha,0)$ has an abelian subgroup $C$ that satisfies $G = AC = BC$ and $A \cap C = B \cap C = G'$.  We now prove a sort of converse.  Note that this is Theorem \ref{mainf}.

\begin{lemma}\label{ten}
Assume Hypothesis 1 with $\beta = 0$.  Then $G =  G(\alpha)$ has an abelian subgroup $C$ so that $G = AC = BC$ and $A \cap C = B \cap C = G'$ if and only if $\alpha$ is isotopic to an associative generalized nonsingular map.
\end{lemma}

\begin{proof}
Suppose $\alpha$ is isotopic to the associative generalized nonsingular map $\alpha_1$.  By Lemma \ref{eight}, we see that $G (\alpha)$ and $G(\alpha_1)$ are isomorphic.  Applying Corollary \ref{threea}, we see that the subgroup $C$ exists as stated.  This proves the desired conclusion.

Conversely, suppose that a group $C$ exists as given in the statement.  By Lemma \ref{three}, there is an additive map $f : V \rightarrow V$ so that $C = B_f$.  Since $C \cap B = G'$, it is not difficult to see that $f (v) \ne 0$ for all $v \in V \setminus \{ 0 \}$.  This implies that $f$ is an isomorphism from $V$ to $V$.  Using Lemma \ref{three}, we see since $C$ is abelian that $\alpha (f(v_1),v_2) = \alpha(f (v_2),v_1)$ for all $v_1, v_2 \in V$.  Define $\alpha_f$ by $\alpha_f (v_1,v_2) = \alpha (f(v_1),v_2)$.  It is not difficult to see that $\alpha_f$ is a symmetric nonsingular map and that $\alpha_f$ is isotopic to $\alpha$.
\end{proof}

Note that Lemmas \ref{five} and \ref{ten} together can be viewed as a generalization of Theorem 3.14 of \cite{ver}, Proposition 4.2 (i) of \cite{hir}, and Lemma 4.3 of \cite{knst1} which proved that a semifield group had at least three abelian subgroups of the maximal possible order if and only if the semifield is isotopic to a commutative semifield.  Combining Lemmas \ref{five} and \ref{ten} with Corollary \ref{threea}, we see that a generalized nonsingular map $\alpha$ is commutative if and only if $\ker{\phi_\alpha} = 1$.

We now look at the relationship between the generalized nonsingular maps arising from different abelian subgroups.



With this in mind, we say that two generalized nonsingular maps $\alpha_1, \alpha_2 : V \times V \rightarrow W$ are {\rm anti-isotopic} if there exist linear isomorphisms $a,b : V \times V$ and $c : W \rightarrow W$ so that $\alpha_2 (b(v_2),a(v_1)) = c (\alpha_1 (v_1,v_2))$ for all $v_1, v_2 \in V$.  Observe that if $G$ is a s.e.s. group with abelian subgroups $A$ and $B$ such that $G = AB$ and $A \cap B = G'$, then $\alpha_{G,A}$ and $\alpha_{G,B}$ are anti-isotopic.  To see this, we take $a$ and $b$ to be the identity maps and $c$ to be the map taking every element to its negative.  Notice that if $\alpha_2$ is anti-isotopic to both $\alpha_1$ and $\alpha_3$, then $\alpha_1$ and $\alpha_3$ are isotopic.   The following should be compared to Proposition 3.2 (2) of \cite{knst1}.

\begin{lemma} \label{twelve}
Suppose $\alpha_1, \alpha_2 : V \times V \rightarrow W$ are generalized nonsingular maps.  If $(a,b,c)$ is an anti-isotopism from $\alpha_1$ to $\alpha_2$, then the map $G(\alpha_1) \rightarrow G(\alpha_2)$ given by $(v,w,z) \mapsto (b(w),a(v),c(\alpha_1 (v,w) - z))$ is an isomorphism.
\end{lemma}

\begin{proof}
Let $f$ be our map and let $g_1 = (v_1,w_1,z_1)$ and $g_2 = (v_2,w_2,z_2)$.  Then $f (g_1 g_2) = (b(w_1 + w_2),a(v_1 + v_2), c(\alpha_1 (v_1+v_2,w_1+w_2) - (z_1 + z_2 + \alpha_1 (v_1,w_2)))$.  On the other hand, $f (g_1) f(g_2) = (b(w_1)+b(w_2), a(v_1)+a(v_2), c(\alpha_1 (v_1,w_1) - z_1) + c(\alpha_1 (v_2,w_2) - z_2) + \alpha_2 (b(w_1),a(v_2))).$  Notice that $\alpha_1 (v_1+v_2,w_1+w_2) = \alpha_1 (v_1,w_1) + \alpha_1 (v_1,w_2) + \alpha_1 (v_2,w_1) + \alpha (v_2,w_2)$, so that third coordinate of $f (g_1 g_2) = c (\alpha_1 (v_1,w_1)) + c(\alpha_1 (v_2,w_1)) + c(\alpha_1 (v_2,w_2)) - c (z_1) - c(z_2)$.  Since $\alpha_2 (b(w_1),a(v_2)) = c(\alpha_1 (v_2,w_1))$, we see that the third coordinate of $f(g_1) f(g_2)$ is $c (\alpha_1 (v_1,w_1) - c(z_1) + c (\alpha_1 (v_2,w_2) - c(z_2) + c(\alpha_1 (v_2,w_1)$, and we see we have the equality needed for $f$ to be an isomorphism.
\end{proof}

Let $G$ be a s.e.s. group with $|G:G'| = p^{2n}$ and $|G'| = p^m$.  Define $\mathcal{A} (G)$ to be the set of abelian subgroups of $G$ with order $p^{m+n}$.  We attach a graph to $\mathcal {A} (G)$ as follows.  We take $\mathcal {A} (G)$ to be the set of vertices.  We put an edge between $A$ and $B$ if $G = AB$ and $A \cap B = G'$.  It is not difficult to see that if there is an edge between $A$ and $B$, then $\alpha_{G,A}$ and $\alpha_{G,B}$ are anti-isotopic.  It follows that if $A$ and $B$ are the same connected component of this graph, then $\alpha_{G,A}$ and $\alpha_{G,B}$ are either isotopic or anti-isotopic.  Notice that if $m > n/2$, then we know that $\mathcal {A} (G)$ is a complete graph.  With this in mind, the following result which is Theorem \ref{maine} is immediate.

\begin{corollary} \label{thirteen}
Suppose $\alpha_1, \alpha_2 : V \times V \rightarrow W$ are generalized nonsingular maps and that $G(\alpha_1)$ and $G(\alpha_2)$ are isomorphic.  If $\mathcal {A} (G)$ has one connected component, then $\alpha_1$ and $\alpha_2$ are either isotopic or anti-isotopic.  In particular, if $m > n/2$, then $\alpha_1$ and $\alpha_2$ are either isotopic or anti-isotopic.
\end{corollary}

We close by noting that Corollary \ref{thirteen} generalizes a similar result for semifield groups that was proved as Theorem 5.1 of \cite{hir} and Theorem 6.6 of \cite{knst1}.

\end{document}